\documentclass{amsart}
\usepackage{amsfonts,amssymb,amscd,amsmath,enumerate,verbatim,calc}
\usepackage[pagebackref=false, colorlinks, linkcolor=blue, citecolor=magenta]{hyperref}
\usepackage{graphicx}
\usepackage{caption}
\newcommand{\bt}{\begin{theorem}}
\newcommand{\et}{\end{theorem}}
\newtheorem{theorem}{Theorem}[section]
\newtheorem{lemma}[theorem]{Lemma}
\newtheorem{proposition}[theorem]{Proposition}
\newtheorem{corollary}[theorem]{Corollary}
\newtheorem{example}[theorem]{Example}

\newtheorem{remark}[theorem]{Remark}
\newtheorem{definition}[theorem]{Definition}

\newcommand{\T}{\mathrm}

\newcommand {\be}{\begin{equation}}
\newcommand {\ee}{\end{equation}}

\def \textit{\it}
\def \bt{\begin{theorem}}
\def \et{\end{theorem}}
\def \bl{\begin{lemma}}
\def \el{\end{lemma}}
\def \bc{\begin{corollary}}
\def \ec{\end{corollary}}
\def \be{\begin{equation}}
\def \ee{\end{equation}}

\def \text{\mbox}

\title{Kirchhoff index and  degree Kirchhof index of  complete multipartite graphs}
\author[R.B. Bapat, R M. Karimi, J.B. Liu]{Ravindra B. Bapat$^{1}$,  Masoud Karimi$^{*,2}$, Jia\textrm{-}Bao  Liu$^{3}$}
\date{}
\begin{document}
\thanks{$^*$Corresponding Author}
\maketitle
\begin{center}
$^1$Indian Statistical Institute,\\ Delhi Centre, 7 S.J.S.S. Marg, \\  New Delhi 110 016, India\\
 {rbb@isid.ac.in}

 $^{2}$ School of mathematical sciences, Anhui university 
  \\Hefei, China \\
 
 $^{2}$Department of Mathematics,
  Bojnourd Branch,\\ Islamic Azad University, Bojnourd, Iran \\
  karimimth@bojnourdiau.ac.ir
  
$^{3}$School of Mathematics and Physics, Anhui Jianzhu University\\ Hefei, 230601, PR China\\
{ liujiabaoad@163.com}
\end{center}

\begin{abstract}
The  Kirchhoff index of a graph  is defined as half of the  sum of all effective resistance distances between any two vertices. Assuming a complete multipartite graph $G$, by methods from linear algebra   we explicitly formulate effective resistance distances between any two vertices of $G$, and  its Kirchhoff index. In rest of paper we explore extremal value of Kirchhoff index for multipartite graphs.  
\end{abstract}

\vspace{3mm}
\section{Introduction}
We consider simple graphs, that is, graphs without loops or parallel edges. For basic
notions in graph theory we refer to \cite{west}, whereas for preliminaries on
graphs and matrices, see \cite{bapat}.
 A graph $G$ is said to be a regular graph if all its vertices have the  same degree. The complete graph of order $p$ is denoted by $K_{p}.$
The disjoint union of graphs $G$ and  $H$ is denoted by $G\cup H$. The complement of $G $ is denoted by $\overline{G},$ which  is a simple graph such  that $G\cup \overline{G}$ is a complete graph. The complete join of  graphs $G$ and $H$ is denoted by $G\vee H$ which is a graph with vertex-set 
$V(G\vee H):=V(G)\cup V(H)$ and edge set 
$$E(G\vee H):=E(G)\cup E(H)\cup\lbrace uv~|~u\in V(G) , v\in V(H) \rbrace.$$ 
A complete $r$-partite graph is a complete join of $r$ empty graphs. We denote $K_{p_1}\vee\dots \vee K_{p_r}$ by $K_{p_1,\dots,p_r}$.
Also for convenience,   by  $K_{ \overline{p}^{a_1}_1,\overline{p}^{a_2}_2,\dots , \overline{p}^{a_s} _s}$  we denote a complete $r$-partite graph with   $a_i$ parts of size $\overline{p}_i$ with $\overline{p}_i< \overline{p}_{i+1}$, for $i=1,\dots, s$. Thus,  $a_1+\dots +a_s=r$ and $a_1\overline{p}_1+\dots +a_s\overline{p}_s=n$ is the number of vertices. Thus
 $K_{p_1,\dots,p_r} $ and $K_{ \overline{p}^{a_1}_1,\overline{p}^{a_2}_2,\dots , \overline{p}^{a_s} _s}$ denote the
same complete $r$-partite graph.

The Laplacian matrix of the graph $G,$ is denoted by $L(G),$ is  defined as $D-A$, where $D$ is the
diagonal matrix with degrees of vertices of $G$ on the diagonal and $A$ is the adjacency matrix of $G.$ 
We denote by $I_n$ and  $J_n$, respectively,  the $n \times n$ identity matrix and the $n \times n$matrix with all ones.
Let $A$ be an $n\times n$ matrix. We will use $I$ and $J$ if the order is clear
from the context.  If $T,S\subseteq\lbrace 1,\dots ,n \rbrace$, then $A[S\vert T]$ will denote the submatrix of $A$ indexed by the rows corresponding to $S$ and the columns corresponding to $T$. The submatrix obtained by deleting the rows in $S$ and  the columns in $T$ will be denoted by $A(S\vert T)$. When $S=\lbrace i\rbrace$ and $T=\lbrace j\rbrace,$ 
$A(S\vert T)$ is denoted by $A(i\vert j)$.
Direct sum of two matrices $A$ and $B$ is written as $A\oplus B$ which is a diagonal block matrix with $A$ and $B$  as block-diagonal entries. The sum of all diagonal entries of a matrix $A$ is denoted by $\T{trace}(A)$. If $A$ is an invertible matrix, by $A^{-1}$ we denote its inverse matrix. Also, by $A^{+}$ we denote the Moore-Penrose inverse of $A$.

We  recall the notion of equitable partition. Suppose $A$  is a symmetric matrix whose rows and columns are indexed by $\{1,\dots, n\}$. Let $X=\{X_1,\dots,X_s\}$ be a partition of $\{1,\dots, n\}$. By definition, $X$ is an equitable partition if 
 $A[X_i|X_j] {\bf 1}= b_{ij} {\bf 1}$ for $i,j=1,\dots, s$. Let  $H$ denote the $n\times s$ matrix whose  $j$-th
 column is the characteristic vector of $X_j,$ for $j=1,\dots , s.$ The $s \times s$
matrix  $B = ((b_{ij}))$   is called the quotient matrix of $A$ with respect to the  partition $X.$
 Then we have $AH=HB$,  and so eigenvalues of $A$ consist of the eigenvalues of $B,$ together with the eigenvalues belonging to the eigenvectors orthogonal to the columns of $H$ (i.e., summing to zero on each part of the partition); see \cite[page 24]{brouwer}.

\section{Resistance distance }
S.V. Gervacio in \cite{Ger}, by using methods and principals in electrical circuits such as
the  principle of elimination and the principle of substitution, explicitly expressed the effective resistance distance between any pair of vertices in the complete multipartite graph.
In this section, we employ methods from linear algebra to obtain the expression. In the process we prove a result
(Theorem \ref{tm1}) which is of independent interest.
\begin{lemma}\label{1}
Suppose that $\tau_n'=[a_1,\dots,a_n]$ and $\sigma_n'=[b_1,\dots,b_n]$ are vectors. If 
 $c_i\neq 0$ for $i=1,\dots, n,$ we have 
$$\T{det}(\tau_n\sigma_n'+\T{diag}(c_1,\dots,c_n))=c_1c_2\dots c_n
\left (1+\frac{a_1b_1}{c_1}+\dots+\frac{a_nb_n}{c_n} \right ).$$
\end{lemma}
\begin{proof}
Put $ B:=\T{diag}(c_1,\dots,c_{n}) $. Then it is not hard to see that
\be\label{eq5}
\T{det}(\tau_{n}\sigma_n' +B)=\left|
  \begin{array}{cccccc}
1 & a_1&\dots &a_{n}\\
 -b_1&&&\\

 \vdots &&B&\\
 -b_n&&&

    \end{array}\right|=\T{det}(B)\left (1+\frac{a_1b_1}{c_1}+\dots+\frac{a_nb_n}{c_n} \right ),
\ee
and the proof is complete.
\end{proof}
\begin{theorem}\label{tm1}
Let $=K_{p_1,p_2,\dots,p_r}$ be a complete $r$-partite graph with $n$ vertices and with $V_1,\dots, V_r$ as its parts.  
Let $A \subset \{1, \ldots, n\}.$ Then  the eigenvalues  of $L(A|A)$ are the roots of the polynomial
\be\label{eq1}
\left (1+\sum_{k=1}^r\dfrac{t_k-p_k}{d_k+p_k-t_k-x}\right )\prod_{k=1}^r(d_k+p_k-t_k-x)(d_k-x)^{p_k-t_k-1},
\ee
where $t_i:=|A\cap V_i|$,  and $d_i$ is the degree of a vertex in $V_i$ for  $i=1,\dots , r$.
\end{theorem}
\begin{proof}
Let $L$ denote the Laplacian of $K_{p_1,p_2,\dots,p_r}$. We can verify that 
$$ L=(d_1I_{p_1}+J_{p_1})\oplus\dots\oplus (d_rI_{p_r}+J_{p_r})-J_n.$$
Thus,
 $$L(A|A)=(d_1I_{p_1}+J_{p_1})(T_1|T_1)\oplus\dots\oplus (d_rI_{p_r}+J_{p_r})(T_r|T_r)-J_n(A|A)$$
where  $T_i:=A\cap V_i.$ We can see that  $\pi=\{V_1-T_1,\dots, V_r-T_r\}$  is an equitable partition with respect to $L(A|A)$ with the quotient matrix
 \be \nonumber
B:=\left[
\begin{array}{ccccc}
d_1&t_2-p_2&t_3-p_3&\dots &t_r-p_r\\
t_1-p_1&d_2&t_3-p_3&\dots &t_r-p_r\\
t_1-p_1&t_2-p_2&d_3&\dots& t_r-p_r\\
\vdots &\vdots &\vdots & \ddots&\vdots\\
t_1-p_1&t_2-p_2&t_3-p_3 &\dots &d_r\\
 \end{array}
\right].
\ee
Now, eigenvalues of $L(A|A)$ consist of $[d_i]^{p_i-t_i-1}$ for $i=1,\dots, r$ and the
roots of $\T{det}(B-xI)=0$. Since 
$$B-xI=\T{diag}(d_1+p_1-t_1-x,\dots,d_r+p_r-t_r-x)+{\bf 1}[t_1-p_1,\dots,t_r-p_r],$$
in view of Lemma \ref{1}, we have 
$$\T{det}(B-xI)=\left (1+\sum_{k=1}^r\dfrac{t_k-p_k}{d_k+p_k-t_k-x}\right )\prod_{k=1}^r(d_k+p_k-t_k-x).$$
So the characteristic polynomial of $L(A|A)$ is 
$$\left (1+\sum_{k=1}^r\dfrac{t_k-p_k}{d_k+p_k-t_k-x}\right )\prod_{k=1}^r(d_k+p_k-t_k-x)(d_k-x)^{p_k-t_k-1}.$$
That completes the proof.
\end{proof}
\begin{remark}
Let $n:=p_1+\dots + p_r$ be the number of vertices of $=K_{p_1,p_2,\dots,p_r}$. Then $d_i=n-p_i$, so we can rewrite  $\T{det}(L(A|A)-xI_n)$ as 
$$\left (1-\sum_{k=1}^r\dfrac{p_k-t_k}{n-t_k}\right )\prod_{k=1}^r(n-t_k)(n-p_k)^{p_k-t_k-1}.$$
\end{remark}
\begin{corollary}
Let $K_{p_1,p_2,\dots,p_r}$ be a complete $r$-partite graph with $n$ vertices. For distinct vertices $u$ and $v$  we have
\begin{equation}\nonumber
\T{det}(L(u,v|u,v))=\left\{
\begin{array}{ll}
n^{r-3}\dfrac{(n-1)(2n-p_l-p_{l'})}{(n-p_l)(n-p_{l'})}\prod_{i=1}^r(n-p_i)^{p_i-1}                                        &  \text{if }  u\in V_l, v\in V_{l'}, l\neq l'\\
n^{r-2}\dfrac{2}{n-p_l}\prod_{i=1}^r(n-p_i)^{p_i-1}     & \text{if }u,v\in V_l. \\
\end{array}\right.
\end{equation}
Moreover,  $\T{det}(L(u|u)) = n^{r-2}\prod_{i=1}^r(n-p_i)^{p_i-1}$ is the number of spanning trees.
\end{corollary}
\begin{proof}
Let $u $ and $v$ be vertices of $K_{p_1,p_2,\dots,p_r}$ and let $A:=\{u,v\}.$ If $i$ and $j$ are chosen  from different parts, say $V_l$ and $V_{l'}$ with $l\neq l'$, then by Theorem \ref{tm1} we have $t_l=t_{l'}=1$ and $t_k=0$ for $k\neq l,l'$. Evaluating (\ref{eq1}) at $x=0$ gives 
$$ \T{det}(L(u,v|u,v))=n^{r-3}\dfrac{(n-1)(2n-p_l-p_{l'})}{(n-p_l)(n-p_{l'})}\prod_{i=1}^r(n-p_i)^{p_i-1}.$$
Now let $u$ and $v$ belong to the same part, say $V_l$. Then, $t_l=2$  and $t_k=0$ for $k\neq l$. Again Evaluating (\ref{eq1}) at $x=0$ gives 
$$ \T{det}(L(u,v|u,v))=n^{r-2}\dfrac{2}{n-p_l}\prod_{i=1}^r(n-p_i)^{p_i-1}.$$
The expression for $\T{det}(L(u|u))$  is obtained by setting $A:=\{u\}$ in
(\ref{eq1}) and putting  $x=0$, and $t_i=1$ when $u\in V_i$, otherwise $t_i=0$ for $i=1, \dots, r.$
\end{proof} 
\begin{definition}
Let $G$ be a connected graph. The  resistance distance between vertices $u$ and $v$ is defined as $\Omega(u,v)=:\dfrac{\T{det}(L(u,v|u,v))}{\T{det}(L(u|u))}$.
\end{definition}
\begin{corollary}[\cite{Ger}, Theorem 5.1]
Let $u\in V_i$ and $v\in V_j$. Then the resistance distance between $u$ and $v$ is given by
\begin{equation}\nonumber
\Omega(u,v)=\left\{
\begin{array}{ll}
\dfrac{(n-1)(2n-p_i-p_{j})}{n(n-p_i)(n-p_{j})}                                 &  \text{if } u\in V_i, v\in V_{j}, i\neq j\\
\dfrac{2}{n-p_i}    & \text{if }u,v\in V_i, u\neq v. \\
0                          &\text{if }u=v.
\end{array}\right.
\end{equation}
\end{corollary}
\begin{example}
The resistance distance matrix  of $K_{2,3,4}$ is as follows
$$
\left[ \begin {array}{ccccccccc} 
0&\frac{2}{7}&{\frac {52}{189}}&{\frac {52}{189}}&{\frac {52}{189}}&{\frac {32}{105}}&{\frac {32}{105}}&{\frac {32}{105}}&{\frac {32}{105}}\\ 
\noalign{\medskip}\frac{2}{7}&0&{\frac {52}{189}}&{\frac {52}{189}}&{\frac {52}{189}}&{\frac {32}{105}}&{\frac {32}{105}}&{\frac {32}{105}}&{\frac {32}{105}}\\
 \noalign{\medskip}{\frac {52}{189}}&{\frac {52}{189}}&0&\frac{1}{3}&\frac{1}{3}&{\frac {44}{135}}&{\frac {44}{135}}&{\frac {44}{135}}&{\frac {44}{135}}\\ 
\noalign{\medskip}{\frac {52}{189}}&{\frac {52}{189}}&\frac{1}{3}&0&\frac{1}{3}&{\frac {44}{135}}&{\frac {44}{135}}&{\frac {44}{135}}&{\frac {44}{135}}\\ 
\noalign{\medskip}{\frac {52}{189}}&{\frac {52}{189}}&\frac{1}{3}&\frac{1}{3}&0&{\frac {44}{135}}&{\frac {44}{135}}&{\frac {44}{135}}&{\frac {44}{135}}\\
 \noalign{\medskip}{\frac {32}{105}}&{\frac {32}{105}}&{\frac {44}{135}}&{\frac {44}{135}}&{\frac {44}{135}}&0&\frac{2}{5}&\frac{2}{5}&\frac{2}{5}\\ 
\noalign{\medskip}{\frac {32}{105}}&{\frac {32}{105}}&{\frac {44}{135}}&{\frac {44}{135}}&{\frac {44}{135}}&\frac{2}{5}&0&\frac{2}{5}&\frac{2}{5}\\ 
\noalign{\medskip}{\frac {32}{105}}&{\frac {32}{105}}&{\frac {44}{135}}&{\frac {44}{135}}&{\frac {44}{135}}&\frac{2}{5}&\frac{2}{5}&0&\frac{2}{5}\\ 
\noalign{\medskip}{\frac {32}{105}}&{\frac {32}{105}}&{\frac {44}{135}}&{\frac {44}{135}}&{\frac {44}{135}}&\frac{2}{5}&\frac{2}{5}&\frac{2}{5}&0
\end {array}
 \right] 
$$
\end{example}
\section{Kirchhoff index}
In the following, we compute the Kirchhoff index of a complete multipartite graph. 
\begin{definition}
The  Kirchhoff index of a connected graph $G$, denoted by $\T{Kf}(G)$,  is defined as half of the summation of all resistance distances between any two vertices, i.e, 
$$ \T{Kf}(G)=\frac{1}{2}\sum_{u,v\in V(G)}\Omega(u,v).$$
\end{definition}
In this regard, there are  two applicable  relations: 
\be
\T{Kf}(G)=n\T{trace}(L^+(G))~~~~\text{and}~~~~\T{Kf}(G)=n\sum_{i=1}^{n-1}\frac{1}{\lambda_i},
\ee\label{Kf}
 where $L^+(G)$ is the Moore-Penrose inverse of the Laplacian of $G$ and $\lambda_1\geq\dots\geq\lambda_{n-1}>\lambda_n=0$ are the Laplacian eigenvalues of $G$.  
\begin{remark}\label{rem1}
Let $G$ be a  connected graph. Then the following facts hold
\begin{itemize}
\item[(1)] $L(G)+L(\overline{G})=nI-J$.
\item[(2)] $(L(G)+J)^{-1}=L^+(G)+\frac{1}{n^2}J$. 
\item[(3)] $(-L(G)+nI)^{-1}=L^+(\overline{G})+\frac{1}{n^2}J$.
\item[(4)] If $\alpha,\beta$ are scalers in such a way that $\alpha I+\beta J$ is nonsingular, then $(\alpha I+\beta J)^{-1}=\frac{1}{\alpha}I-\frac{\beta}{\alpha(\alpha+n\beta)}J$.
\end{itemize}
\end{remark}
\begin{proposition}\label{prop1}
Let $G_i$ be a connected graph of order $p_i$, for $i=1,\dots,r$.
Then
 $$L^+(\vee_{i=1}^r G_i)=-\frac{1}{n^2}J+\oplus_{i=1}^r (nI_{p_i}-L(\overline{G_i}))^{-1}.$$
 In particular,  If 
$\lambda^i_1\geq\dots\geq\lambda^i_{p_i-1}>\lambda^i_{p_i}=0$
  are the Laplacian eigenvalues of $G_i$, then 
$$0,[n]^{r-1},n-p_i+\lambda_i^j$$
for $ i=1,\dots,p_j-1$ and $j=1,\dots,r$ are the Laplacian eigenvalues of $\vee_{i=1}^r G_i$, where $n=p_1+\dots+p_r$.
\end{proposition}
\begin{proof}
We see that
 $$L(\vee_{i=1}^r G_i)+J=\oplus_{i=1}^ r(J_{p_i}+(n-p_i)I_{p_i}+L(G_i))=\oplus_{i=1}^r(nI-L(\overline{G_i})). $$
 In view of Remark \ref{rem1}, 
 $$(L(\vee_{i=1}^r G_i)+J)^{-1}=L^+(\vee_{i=1}^r G_i)+\frac{1}{n^2}J=\oplus_{i=1}^r(nI_{p_i}-L(\overline{G_i}))^{-1}.$$
 Therefore,  $L^+(\vee_{i=1}^r G_i)=-\frac{1}{n^2}J+\oplus_{i=1}^r (nI_{p_i}-L(\overline{G_i}))^{-1}.$
 
 In particular,  in view of Remark \ref{rem1}(1) and \cite[Lemma 4.5]{bapat}, the Laplacian eigenvalues of $\overline{G_i}$ are $\lambda=0$ along with $p_i-\lambda^i_j$, for $j=1,\dots,p_i$. Thus, the
 eigenvalues of $ L(\vee_{i=1}^r G_i)+J $  comprise of the eigenvalues of $ nI-L(\overline{G_i}) $, for $i=1,\dots, r$, that is, $\{n\}^r$, $n-p_i+\lambda^i_j$, for $j=1,\dots,p_i$ and $i=1\dots, r$. Again, in view of   \cite[Lemma 4.5]{bapat}, the eigenvalues of $ L(\vee_{i=1}^r G_i)$ are $0,[n]^{r-1},n-p_i+\lambda_i^j$.
for $ i=1,\dots,p_j-1$ and $j-1,\dots,r$ are the Laplacian eigenvalues of $\vee_{i=1}^r G_i$, where $n=p_1+\dots+p_r$.
\end{proof}
\begin{corollary}\label{cor1}
Under the assumptions of Proposition \ref{prop1}, we have $$\T{Kf}(\vee_{i=1}^r G_i))=r-1+n\sum_{i=1}^r\sum_{j=1}^{p_i-1}\frac{1}{n-p_i+\lambda_j^i}.$$
\end{corollary}
\begin{proof}
It is enough to employ  (\ref{Kf}) together with the result of Proposition \ref{prop1}. We see
 \begin{align*}
 \T{Kf}(\vee_{i=1}^r G_i))&=n\sum_{i=1}^{n-1}\frac{1}{\lambda_i}=n(\frac{r-1}{n}+\sum_{i=1}^r\sum_{j=1}^{p_i-1}\frac{1}{n-p_i+\lambda_j^i}).
  \end{align*}
and the proof is complete. \end{proof}
\begin{corollary}\label{cor2}
Let $G:=K_{p_1,\dots,p_r}$ be a complete multipartite graph. Then
$$\T{Kf}(\vee_{i=1}^r G_i))=r-1+n\sum_{i=1}^r\frac{p_i-1}{n-p_i}.$$
\end{corollary}
\begin{proof}
In view of Corollary \ref{cor1}, we need only to consider $G_i$ as empty graph of order $p_i$. Apparently, they are $0$-regular graphs. So, in this case we have $\lambda_i^j=0$. substitute  in Corollary \ref{cor1} to get 
$$\T{Kf}(\vee_{i=1}^r G_i))=r-1+n\sum_{i=1}^r\sum_{j=1}^{p_i-1}\frac{1}{n-p_i}=r-1+n\sum_{i=1}^r\frac{p_i-1}{n-p_i},$$
completing the proof.
\end{proof}
An alternative way to derive the Kirchhoff index is as follows. 
Since $\T{Kf}(G)=n\T{trace}(L^+(G))$, thus in view of Proposition \ref{prop1}, 
 \begin{align}
 \T{Kf}(K_{p_1,\dots,p_r})&=n\T{trace}(-\frac{1}{n^2}J+\oplus_{i=1}^r (nI_{p_i}-L(K_{p_i}))^{-1})\nonumber \\
 =&n\T{trace}(-\frac{1}{n^2}J+\oplus_{i=1}^r (nI_{p_i}-L(K_{p_i}))^{-1})\nonumber\\
  =&n\T{trace}(-\frac{1}{n^2}J+\oplus_{i=1}^r ((n-p_i)I_{p_i}+J_{p_i})^{-1})\nonumber\\
 =&n\T{trace}(-\frac{1}{n^2}J+\oplus_{i=1}^r (\frac{1}{n-p_i}I-\frac{1}{n(n-p_i)}J))\nonumber\\
 =&n(\frac{-1}{n}+\sum_{i=1}^r(\frac{p_i}{n-p_i}-\frac{p_i}{n(n-p_i)}))\nonumber\\
 =&-1+n\sum_{i=1}^r\frac{p_i}{n-p_i}(1-\frac{1}{n})\nonumber\\
 =&-1+(n-1)\sum_{i=1}^r\frac{p_i}{n-p_i}\nonumber\\
 =&-1-r(n-1)+n(n-1)\sum_{i=1}^r\frac{1}{n-p_i}.\label{2}
 \end{align}
\begin{example}
$ \T{Kf}(K_{2,2,2,3,3,5,7})=-1+23(\frac{6}{22}+\frac{6}{21}+\frac{5}{19}+\frac{7}{17})\approx 27.367$.
\end{example}
\section{Extremal Kirchhoff index}
In this section, motivated by \cite{Liu1},  we explore  within all complete $r$-partite graphs with $n$ vertices,  which  have maximal or minimal Kirchhoff index. 
By routine and straightforward calculations we obtain the following lemma. 
\begin{lemma}\label{lm2.2}
Let $x\geq y > \alpha>0$. Then 
\begin{itemize}
\item[(1)] $\dfrac{1}{x+\alpha}+\dfrac{1}{y-\alpha}\geq \dfrac{1}{x}+\dfrac{1}{y}$
\item[(2)] $\dfrac{1}{x- \alpha}+\dfrac{1}{y+\alpha}\leq \dfrac{1}{x}+\dfrac{1}{y}$ if $x-y\geq \alpha$. 
\end{itemize}
\end{lemma}
\begin{theorem}
Let $p_1\geq\dots \geq p_r$ and $n=p_1+\dots +p_r$.
$ \T{Kf}(K_{p_1,\dots,p_r})$ has maximal Kirchhoff index if $ K_{p_1,\dots,p_r}=K_{(n-r+1),1^{r-1}} $
Also,  $\T{Kf}(K_{p_1,\dots,p_r})$ has minimal Kirchhoff index if $ K_{p_1,\dots,p_r}=K_{p^k,(p-1)^{r-k}}$ where $p=\lfloor\dfrac{n}{r}\rfloor+1$ and $ k=n-r\lfloor\dfrac{n}{r}\rfloor$.
\end{theorem}
\begin{proof}
In view of (\ref{2}),  parameters $p_1,\dots, p_r$ with extremal Kirchhoff index  can be derived from those which give extremal values of $f(p_1,\dots, p_r):=\sum_{i=1}^r\frac{1}{n-p_i}$ where $p_1+\dots+ p_r=n$. Employing Lemma \ref{lm2.2}, we have 
 \begin{eqnarray*}
f(p_1,\dots, p_r) &\leq&  f(p_1+p_r-1,p_2\dots,p_{r-1}, 1)\\
&\leq& f(p_1+p_r+p_{r-1}-2,p_2,\dots,1, 1)\\
&\vdots&\\
& \leq & f(p_1+p_r+\dots+p_2-r+1,1,\dots,1)\\
& =& f(n-r+1,1,\dots, 1)=\dfrac{1}{r-1}+\dfrac{r-1}{n-1}
\end{eqnarray*}
 So the maximal Kirchhoff index is attained when  $p_1=n-r+1$ and $p_i=1$ for $i=2,\dots,r$.
 In this case we have 
 $$\T{max}_{p_1+\dots+p_r=n}(\T{Kf}(K_{p_1,\dots,p_r}))=-1-r(n-1)+n(n-1)(\dfrac{1}{r-1}+\dfrac{r-1}{n-1}).$$
 
 As to the minimal Kirchhoff index, let $f(p_1,\dots, p_r):=\sum_{i=1}^r\frac{1}{n-p_i}$ be minimal. We claim that   $|p_i-p_j|\leq1$  for $1\leq i,j\leq r$. In  contrast, let there exist $p_i$ and $p_j$ such that $p_i-p_j\geq2$. Then, 
 $$f(p_1,\dots,p_i-1,\dots,p_j+1,\dots, p_r)-f(p_1,\dots, p_r)=\dfrac{1}{n-p_i+1}+\dfrac{1}{n-p_j-1}-\dfrac{1}{p_i}-\dfrac{1}{p_j}.$$
 This together with Lemma \ref{lm2.2} implies that $ f(p_1,\dots,p_i-1,\dots,p_j+1,\dots, p_r)\leq f(p_1,\dots, p_r) $. This violates the minimality of $f(p_1,\dots, p_r):=\sum_{i=1}^r\frac{1}{n-p_i}$. So the maximal Kirchhoff index is attained when $|p_i-p_j|\leq 1$, for $1\leq i,j\leq r$. In this regard let $f(p_1,\dots, p_r):=\sum_{i=1}^r\frac{1}{n-p_i}$ be minimal. In view of  the discussion above, for some integer $k$ we  assume that $p:=p_1=\dots=p_k$ and $p_{k+1}=\dots=p_r=p-1$. Since $ p_1+\dots+ p_r =n$, we see that $kp+(r-k)(p-1)$. So, $p=\lfloor\dfrac{n}{r}\rfloor+1$ and $ k=n-r\lfloor\dfrac{n}{r}\rfloor$. In this case we have $$\T{min}_{p_1+\dots+p_r=n}(\T{Kf}(K_{p_1,\dots,p_r}))=-1-r(n-1)+n(n-1)(\dfrac{k}{n-p}+\dfrac{r-k}{n-p+1}),$$
and the proof is complete.
\end{proof}
\begin{example}
Let $n=24$ and $r=7$. Then we see that $k=3$, $p=4$, so 
$$25.943\approx\T{Kf}(K_{4,4,4,3,3,3,3})\leq \T{Kf}(K{p_1,\dots,p_7})\leq 74= \T{Kf}(K_{18,1,1,1,1,1,1}).$$
\end{example}
\section{Degree Kirchhoff index}
Provided that the graph $G$ has no isolated vertices, the normalized Laplacian matrix is defined as $\mathcal{L}(G):=D^{\frac{-1}{2}}L(G)D^{\frac{-1}{2}},$ where $L (G)$ is the  standard Laplacian matrix of $G$. Motivated by the notion of  Kirchhoff index, Chen and Zhang  \cite{chen} defined  another graph invariant, denoted  by $\T{Kf}'(G)  $, which is called the {\it degree Kirchhoff index}.  
\begin{definition}\label{KF'}
The degree Kirchhoff index of the graph $G$ is defined as $$\T{Kf}'(G):=\sum_{i<j} d_id_j\Omega(i,j),$$ where $\Omega(i,j)$ is the resistance distance between vertices $i$ and $j$. 
\end{definition}

We refer to  \cite{Bozkurt, ASHRAFI, Liu2, Palacios, Palacios1} for some results about this topological index. 
\begin{lemma}\label{lm2.17}
For the graph $G $ with $m$ edges, $\T{Kf}'(G)=2m\T{trace}(\mathcal{L}^+(G))$. 
\end{lemma}
\begin{proof}
It is known that  if $\lambda_1,\dots,\lambda_t$ are the nonzero eigenvalues of a given matrix $A$, then $\T{trace}(A^+)=\frac{1}{\lambda_1}+\dots +\frac{1}{\lambda_t}$. By \cite[Theorem 2.5]{chen}, we  see that if $\mu_1\geq\dots \geq\mu_{n-1}>\mu_n=0$ are the eigenvalues of $ \mathcal{L}(G) $, then  $\T{Kf}'(G)=2m\sum_{i=1}^{n-1}\frac{1}{\mu_i}$. Thus, Definition \ref{KF'} gives $\T{Kf}'(G)=2m\T{trace}(\mathcal{L}^+(G))$.
\end{proof}
We proceed to obtain the degree Kirchhoff index of complete multipartite graphs.
\begin{theorem}\label{tm2.19}
Let $K_{p_1,\dots,p_r}$ be a complete $r$-partite graph with $n$ vertices and $m$ edges. Then, $$\T{Kf}'(K_{p_1,\dots,p_r})=-(\dfrac{2m}{n})^2+2m(n-1).$$
\end{theorem}
\begin{proof}
By Lemma \ref{lm2.17} , we need only to find $\T{trace}(\mathcal{L}^+(K_{p_1,\dots,p_r}))$. Since  $\mathcal{L}(K_{p_1,\dots,p_r})=D^{\frac{-1}{2}}L(K_{p_1,\dots,p_r})D^{\frac{-1}{2}}$  we get by  Proposition \ref{prop1} 
$$\mathcal{L}^+(K_{p_1,\dots,p_r})=D^{\frac{1}{2}}L^{+}(K_{p_1,\dots,p_r})D^{\frac{1}{2}}.$$
Let $P_i$  denote part of size $p_i$ in $ K_{p_1,\dots,p_r} $.  Thus, 
\begin{align*}
\T{trace}(\mathcal{L}^+(K_{p_1,\dots,p_r}))=&\sum_{i=1}^r\sum_{v\in P_i}d_iL^{+}(K_{p_1,\dots,p_r})_{vv}\\
=&\sum_{i=1}^rp_id_i(-\dfrac{1}{n^2}+\T{trace}((nI-L(K_{p_i}))^{-1})))\\
=&\sum_{i=1}^rp_id_i(-\dfrac{1}{n^2}+\T{trace}(\frac{1}{n-p_i}I-\frac{1}{n(n-p_i)}J)))\\
=&\sum_{i=1}^rp_id_i(-\dfrac{1}{n^2}+\frac{1}{n-p_i}-\frac{1}{n(n-p_i)}).
\end{align*}
Since $d_i=n-p_i$  after simplifying we get $ \T{trace}(\mathcal{L}^+(K_{p_1,\dots,p_r}))=-\dfrac{2m}{n^2}+n-1 $ where $m$ is the number of edges. So
  $\T{Kf}'(K_{p_1,\dots,p_r})=-(\dfrac{2m}{n})^2+2m(n-1)$ as desired.
\end{proof}
\begin{example}
Cocktail party graph is the $r$-partite graph $K_{2,\dots,2}$. So $n=2r$ and $m=2r(r-1)$. Thus we have $\T{Kf}'(K_{2,\dots,2})=4(2r^3-4r^2+3r-1)$.
\end{example}
\begin{corollary} \label{cor1}
$ \T{Kf}'(K_{p_1,\dots,p_r}) $ is an increasing  function on number of edges. Also we have 
$$\T{Kf}'(K_{n-r+1,1^{r-1}})\leq\T{Kf}'(K_{p_1,\dots,p_r})\leq\T{Kf}'(K_{p^k,(p-1)^{r-k}})$$
where $p=\lfloor\dfrac{n}{r}\rfloor+1$ and $ k=n-r\lfloor\dfrac{n}{r}\rfloor$. 
\end{corollary}
\begin{proof}
In view of Theorem \ref{tm2.19} we have $ \T{Kf}'(K_{p_1,\dots,p_r})=\dfrac{-1}{n^2}(2m)^2+2m(n-1)$. Since the function  $f(x)=\dfrac{-1}{n^2}x^2+x(n-1)$ is increasing in the interval $ (0, \frac{1}{2}n^2(n-1)]$, and that $2m\leq \frac{1}{2}n^2(n-1)$, we deduce that $ \T{Kf}'(K_{p_1,\dots,p_r}) $ is an increasing  function on number of edges. Furthermore, since $2m=n^2-\sum_{i=1}^rp_i^2$, we can verify that among all complete $r$-partite graphs with $n$ vertices the graph  $ K_{n-r+1,1^{r-1}} $ has the least number of edges and the graph $K_{p^k,(p-1)^{r-k}}$ attains the maximum  number of edges where $p=\lfloor\dfrac{n}{r}\rfloor+1$ and $ k=n-r\lfloor\dfrac{n}{r}\rfloor$. 
\end{proof}
As  observed in Corollary \ref{cor1}, degree Kirchhoff index of complete multipartite graphs can be interpreted as a monotone function of the number of edges. But in the ordinary Kirchhof index  we do not observe such a monotonicity,
see Table \ref{tab1} and Table \ref{tab2}:
\begin{table}[th]
$$ \begin {array}{r|rrrrrrrr} 
n=9, r=3&K_{1, 1, 7}& K_{1, 2, 6}& K_{1, 3, 5}& K_{1, 4, 4}& K_{2, 2, 5}& K_{2, 3, 4}&K_{3, 3, 3}\\
\hline
m=|E|&15& 20& 23& 24& 24& 26& 27\\
\T{Kf'}&228.89\nearrow& 300.25\nearrow& 341.88\nearrow& 355.56\nearrow& 355.56\nearrow& 382.62\nearrow& 396\\
\T{Kf}&29\searrow& 18.286\searrow& 14\searrow& 12.800\nearrow& 13.571\searrow& 11.686\searrow& 11\\
\end {array}
$$
\caption{All complete $3$-partite graphs with $9$ vertices}
\label{tab1}
\end{table}
\begin{table}[th]
{\Tiny
$$
  \begin {array}{lllllllllll} 
K_{1^8, 7}& K_{1^7, 2, 6}& K_{1^7, 3, 5}& K_{1^7, 4, 4}& K_{1^6,2^2, 5}& K_{1^6,2,3,4}&K_{1^6,3^3}&K_{1^5 ,2^3, 4}&K_{1^5 ,2^2, 3^ 3}&K_{1^4,2^4, 3}&K_{1^3,2^6}\\
\hline
84& 89&92& 93& 93& 95&96& 96& 97& 98& 99\\
 2226.6\nearrow& 2351.2\nearrow& 2425.5\nearrow& 2450.2\nearrow& 2450.2\nearrow &2499.6\nearrow & 2524.2\nearrow & 2524.2\nearrow & 2548.7\nearrow & 2573.3\nearrow& 2597.8\\ 
 19.250\searrow& 17.487\searrow& 16.5\searrow & 16.182\nearrow & 16.308\searrow & 15.745\searrow & 15.500\nearrow & 15.552\searrow & 15.308\searrow & 15.115\searrow & 14.923\\
 \end {array} 
$$}
\caption{All complete $9$-partite graphs with $15$ vertices}
\label{tab2}
\end{table}
\section*{Acknowledgments}
The third author  was partially supported by National Natural Science Foundation of China grant No. 11601006.

\end{document}